\documentclass[final]{siamltex}

\usepackage{latexsym, amssymb, amsmath}
\usepackage{cancel}
\usepackage{graphicx}
\newtheorem{example}[theorem]{Example}
\let\oldexample\example
\renewcommand{\example}{\oldexample\normalfont}
\newtheorem{remark}[theorem]{Remark}
\let\oldremark\remark
\renewcommand{\remark}{\oldremark\normalfont}
\usepackage{mathdots}
\usepackage[usenames, dvipsnames]{color}

\title{A comparison of eigenvalue condition numbers for matrix polynomials\thanks{This work was partially supported by the Ministerio de Econom\'ia, Industria y Competitividad (MINECO) of Spain
	through grants MTM2015-65798-P and MTM2017-90682-REDT. The research of	L. M. Anguas is funded by the ``contrato predoctoral'' BES-2013-065688 of MINECO.}}

\author{Luis Miguel  Anguas\thanks{Departamento de Matem\'aticas, Universidad Carlos III de Madrid,
		Avda.\ Universidad 30, 28911 Legan\'es, Spain ({\tt languas@math.uc3m.es}).} \and Mar\'ia Isabel Bueno\thanks{Department of Mathematics and College of Creative Studies,
University of California, Santa Barbara, CA 93106, USA ({\tt mbueno@math.ucsb.edu}).}, \and Froil\'an M.  Dopico\thanks{Departamento de Matem\'aticas, Universidad Carlos III de Madrid,
Avda.\ Universidad 30, 28911 Legan\'es, Spain ({\tt dopico@math.uc3m.es}).} }

\begin{document}

\maketitle

\begin{abstract}
In this paper, we consider the different eigenvalue condition numbers for matrix polynomials used in the literature and we compare them. One of these condition numbers is a generalization of the Wilkinson condition number for the standard eigenvalue problem. This number has the disadvantage of only being defined for finite eigenvalues. In order to give a unified approach to all the eigenvalues of a matrix polynomial, both finite and infinite, two (homogeneous) condition numbers have been defined in the literature. In their definition, very different approaches are used.  One of the main goals of this note is to show that, when the matrix polynomial has a moderate degree, both homogeneous numbers are essentially the same and one of them provides a geometric interpretation of the other. We also show how the homogeneous condition numbers compare with the ``Wilkinson-like'' eigenvalue condition number and how they extend this condition number to zero and infinite eigenvalues.
\end{abstract}



\pagestyle{myheadings}
\thispagestyle{plain}
\markboth{L. M. Anguas, M.I. Bueno, and F.M. Dopico}{A comparison of eigenvalue condition numbers for matrix polynomials}

%%%%%%%%%%%%%%%%%%%%
\section{Introduction}\label{sectintro}
Let $\mathbb{C}$ denote the field of complex numbers. A square matrix polynomial of grade $k$ can be expressed (in its non-homogeneous form) as
\begin{equation}\label{pol}
P(\lambda)= \sum_{i=0}^k  \lambda^i B_i , \quad B_i\in \mathbb{C}^{n\times n},
\end{equation}
where the matrix coefficients, including $B_k$, are allowed to be the zero matrix. In particular, when $B_k \neq 0$, we say that $P(\lambda)$ has degree $k$. Throughout the paper, the grade of every matrix polynomial $P(\lambda)$ will be assumed to be its degree unless it is specified otherwise.

The (non-homogeneous) polynomial eigenvalue problem (PEP) associated with a regular matrix polynomial $P(\lambda)$, (that is, $det(P(\lambda))\neq0$) consists of finding scalars $\lambda_0\in \mathbb{C}$ and nonzero vectors {$x,y\in\mathbb{C}^n$} satisfying
$$P(\lambda_0)x=0 \ \text{and}\ y^*P(\lambda_0)=0.$$
The vectors $x$ and $y$ are called, respectively, a right and a left eigenvector of $P(\lambda)$ corresponding to the eigenvalue $\lambda_0$. In addition, $P(\lambda)$ may have infinite eigenvalues. We say that $P(\lambda)$ has an infinite eigenvalue if 0 is an eigenvalue of the reversal of $P(\lambda)$, where the reversal of a matrix polynomial $P(\lambda)$ of grade $k$ is defined as
\begin{equation}
\label{reversal}
rev(P(\lambda)):=\lambda^k P(1/\lambda) \, .
\end{equation}

The numerical solution of the PEP has received considerable attention from many research groups in the last two decades and, as a consequence, several condition numbers for simple eigenvalues of a matrix polynomial $P(\lambda)$ have been defined in the literature to determine the sensitivity of  these eigenvalues to perturbations in the coefficients of $P(\lambda)$  \cite{Ded-first, DedTis2003,Tis2000}. One of these condition numbers is a natural generalization of the Wilkinson condition number for the standard eigenproblem (see Definition \ref{rel-cond}). A disadvantage of this eigenvalue condition number is that it is not defined for infinite eigenvalues. Then, in order to study the conditioning of all the eigenvalues of a matrix polynomial, other condition numbers are considered in the literature. These condition numbers assume that the matrix polynomial is expressed in homogeneous form, that is,
 \begin{equation}
 \label{homform}
 P(\alpha,\beta)=\sum_{i=0}^k\alpha^i\beta^{k-i}B_i.
\end{equation}

If $P(\alpha, \beta)$ is regular, we can consider the corresponding homogeneous PEP that consists in finding pairs of scalars $(\alpha_0,\beta_0)\neq(0,0)$ and nonzero vectors $x,y\in\mathbb{C}^n$ such that
\begin{equation}\label{hom-eig}
P(\alpha_0,\beta_0)x=0 \quad \ \text{and}\ \quad y^*P(\alpha_0,\beta_0)=0.
\end{equation}
We note that the  pairs $(\alpha_0, \beta_0)$ satisfying  (\ref{hom-eig}) are those for which $det(P(\alpha_0,\beta_0))=0$ holds. Notice that $(\alpha_0,\beta_0)$ satisfies $det(P(\alpha_0,\beta_0))=0$ if and only if $det(P(c\alpha_0,c\beta_0))=0$ for any nonzero complex number $c$. Therefore, it is natural to define an eigenvalue of $P(\alpha,\beta)$ as any line in $\mathbb{C}^2$ passing through the origin consisting of solutions of $det(P(\alpha,\beta))=0$. For simplicity, we denote such a line, i.e., an eigenvalue of $P(\alpha,\beta)$, as $(\alpha_0,\beta_0)$ and by $[\alpha_0,\beta_0]\neq (0,0)$ a specific representative of this eigenvalue. The vectors $x,y$ in (\ref{hom-eig}) are called, respectively, a right and a left eigenvector of $P(\alpha,\beta)$ corresponding to the eigenvalue $(\alpha_0,\beta_0)$. For $\beta\neq 0$, we can define $\lambda =\alpha/\beta$ and find a relationship between the homogeneous and the non-homogeneous expressions of a matrix polynomial $P$ of degree $k$  as follows:
$$P(\alpha, \beta) = \beta^kP(\lambda).$$

We note also that, if  $(y, \lambda_0,x)$ is a solution of the non-homogeneous PEP, i.e., an eigentriple of the non-homogeneous PEP, then $(y, (\alpha_0,\beta_0),x)$ is a solution of the corresponding homogeneous PEP, for any ${[\alpha_0,\beta_0]\neq (0,0)}$ such that $\lambda_0=\alpha_0/\beta_0$, including $\lambda_0=\infty$ for $\beta_0=0$.

Two homogeneous  eigenvalue condition numbers (well defined  for all the eigenvalues of  $P(\alpha,\beta)$, finite and infinite) have been presented in the literature. One of them is a natural generalization of the condition number defined by Stewart and Sun in \cite[Chapter VI, Section 2.1]{Stewart} for the eigenvalues of a pencil. This condition number is defined in terms of the chordal distance between two lines in $\mathbb{C}^2$ (see Definition \ref{cord-cond}).

The  other homogeneous eigenvalue condition number is defined as the norm of a differential map \cite{Ded-first, DedTis2003}. The  definition of this condition number is very involved and less intuitive than the definition of the other condition numbers that we consider in this paper. For an explicit formula for this condition number, see Theorem \ref{teorhomcondnumb}.

In this paper we address the following natural questions:

\begin{itemize}
\item how are the two homogeneous eigenvalue condition numbers related? Are they equivalent?
\item if $\lambda_0$ is a finite nonzero eigenvalue of a matrix polynomial $P(\lambda)$ and $(\alpha_0, \beta_0)$ is the associated eigenvalue of $P(\alpha, \beta)$ (that is, $\lambda_0=\alpha_0/\beta_0$), how are the (non-homogeneous) absolute and relative eigenvalue condition numbers of $\lambda_0$ and the (homogeneous) condition numbers of $(\alpha_0, \beta_0)$ related? Are these two types of condition numbers equivalent in the sense that $\lambda_0$ is ill-conditioned if and only if $(\alpha_0, \beta_0)$ is ill-conditioned?
\end{itemize}

Partial answers to these questions are scattered in the literature written in an implicit way so that they seem to be unnoticed by most researchers in Linear Algebra. Our goal is to present a complete and explicit answer to these questions. More precisely, we provide an exact relationship between the two homogeneous eigenvalue condition numbers and we use this relationship to prove that they are equivalent. Also, we obtain exact relationships between each of the non-homogeneous (relative and absolute) and the homogeneous eigenvalue condition numbers. From these relationships we prove that the non-homogeneous condition numbers are always larger than the homogeneous condition numbers. This means that non-homogenous eigenvalues $\lambda_0$ are always more sensitive to perturbations than the corresponding homogeneous ones $(\alpha_0,\beta_0)$, which is natural since $\lambda_0 = \alpha_0/\beta_0$. Moreover, we will see that non-homogeneous eigenvalues with large or small moduli have much larger non-homogeneous than homogeneous condition numbers. Thus, in these cases, $(\alpha_0,\beta_0)$ can be very well-conditioned and $\lambda_0$ very ill-conditioned. In the context of this discussion, it is important to bear in mind that in most applications of PEPs the quantities of interest are the non-homogeneous eigenvalues, and not the homogeneous ones.

The paper is organized as follows: Section \ref{sec2} includes the definitions and expressions of the different condition numbers that are used in this work. In section \ref{seccomp}, we establish relationships between the condition numbers introduced in section \ref{sec2}, and in section \ref{secgeom} we present a geometric interpretation of these relationships. Section \ref{secgeom} also includes a study of the computability of small and large eigenvalues. Finally, some conclusions  are discussed in section \ref{secfinal}.

%%%%%%%%%%%%%%%%%%%%%%%%%%%
\section{Eigenvalue condition numbers of matrix polynomials}\label{sec2}
In this section we recall three  eigenvalue condition numbers used in the literature and discuss some of the advantages and disadvantages of each of them. Before recalling their definition, we present some notation that will be used throughout the paper.

Let $a$ and $b$ be two integers. We define
$$a:b = \left \{ \begin{array}{cc} a, a+1, a+2, \ldots, b, & \textrm{if $a\leq b$,} \\ \emptyset, & \textrm{if $a >b$.}\end{array} \right. $$
For any matrix $A$, $\|A\|_2$ denotes its spectral or 2-norm, i.e., its largest singular value \cite{Stewart}. For any vector $x$, $\|x\|_2$ denotes its standard Euclidean norm, i.e., $\|x\|_2=(x^*x)^{1/2}$, where the operator $()^*$ stands for the conjugate-transpose of $x$.

%%%%%%%%%%%%%%%%%%%
\subsection{Non-homogeneous eigenvalue condition numbers}

Next we recall the definition of  two versions (absolute and relative) of a normwise eigenvalue condition number introduced in \cite{Tis2000}.
\begin{definition}\label{rel-cond}
Let  $\lambda_0$ be a simple, finite eigenvalue of a regular matrix polynomial $P(\lambda)=\sum_{i=0}^k \lambda^i B_i$ of grade $k$  and let $x$ be a right  eigenvector of $P(\lambda)$ associated with $\lambda_0$.  We define the \emph{normwise absolute  condition number} $\kappa_a(\lambda_0,P)$  of $\lambda_0$ by
\begin{align*}
\kappa_{a}(\lambda_0,P) := \lim_{\epsilon\to0} \sup\bigg\{  \frac{|\Delta \lambda_0
|}{\epsilon} : [P(\lambda_0+ \Delta \lambda_0) + \Delta P    &
(\lambda_0+ \Delta \lambda_0)](x + \Delta x) =0,\\
&  \|\Delta B_{i} \|_{2} \leq\epsilon\; \omega_{i}, i=0:k \bigg\},
\end{align*}
where  $\Delta P(\lambda)=\sum_{i=0}^k \lambda^i \Delta B_i$ and $\omega_i$, $i=0:k$, are nonnegative weights that allow flexibility in how the perturbations of $P(\lambda)$ are measured.

For $\lambda_0 \neq 0$,  we define the \emph{normwise relative condition number} $\kappa_r(\lambda_0,P)$  of $\lambda_0$ by
\begin{align*}
\kappa_{r}(\lambda_0,P) := \lim_{\epsilon\to0} \sup\bigg\{\frac{|\Delta \lambda_0
|}{\epsilon|\lambda_0|} : [P(\lambda_0+ \Delta \lambda_0) + \Delta P    &
(\lambda_0+ \Delta \lambda_0)](x + \Delta x) =0,\\
&    \|\Delta B_{i} \|_{2} \leq\epsilon\; \omega_{i}, i=0:k \bigg\}.
\end{align*}

\end{definition}

We will refer to $\kappa_a(\lambda_0, P)$ and $\kappa_r(\lambda_0, P)$, respectively,  as the \emph{absolute and relative  non-homogeneous eigenvalue condition numbers.} Note that the absolute non-homogeneous condition number is not defined for infinite eigenvalues while the relative condition number is not defined for zero or infinite eigenvalues.

\begin{remark}\label{weights}
In the definitions of $\kappa_a(\lambda_0, P)$ and $\kappa_r(\lambda_0, P)$,  the weights $\omega_i$ can be chosen in different ways. The most common ways are: 1) $\omega_i=\|B_{i}\|_2$ (relative coefficient-wise perturbations); 2) $\omega_i=\max \limits_{ i=0: k}\{\|B_i\|_2\}$ (relative perturbations with respect to the norm of $P(\lambda)$); 3) $\omega_i = 1$ (absolute perturbations).
\end{remark}

An explicit formula for each of the non-homogeneous eigenvalue  condition numbers presented above was obtained by Tisseur in \cite{Tis2000}.

\begin{theorem}
	\label{teorTis}
Let $P(\lambda)$ be a regular matrix polynomial of grade $k$. Let $\lambda_0$ be a simple, finite eigenvalue of $P(\lambda)$ and let $x$ and $y$ be, respectively, a right and a left eigenvector of $P(\lambda)$ associated with $\lambda_0$. Then, % normwise relative condition number $\kappa_r(\lambda_0,P)$ is given by
	$$\kappa_a(\lambda_0,P)=\frac{(\sum_{i=0}^k |\lambda_0|^i \omega_i)\|y\|_2\|x\|_2}{|y^*P'(\lambda_0)x|},$$
	where $P'(\lambda)$ denotes the derivative of $P(\lambda)$ with respect to $\lambda$. For $\lambda_0 \neq 0$,
	$$ \kappa_r(\lambda_0,P)=\frac{(\sum_{i=0}^k |\lambda_0|^i \omega_i)\|y\|_2\|x\|_2}{|\lambda_0||y^*P'(\lambda_0)x|}.$$
\end{theorem}

The following technical result will be useful for the comparison of the non-homogeneous condition numbers introduced above and the homogeneous condition numbers that we introduce in the next subsection. We will use the concept of reversal of a matrix polynomial defined in \eqref{reversal}.

\begin{lemma}\label{kappa-a-r}
Let $P(\lambda)=\sum_{i=0}^k \lambda^i B_i$ be a regular matrix polynomial. Let $\lambda_0$ be a simple, nonzero, finite eigenvalue of $P(\lambda)$. Then,
$$\kappa_a\left (\frac{1}{\lambda_0}, rev P \right) =\frac{ \kappa_r(\lambda_0,  P )}{|\lambda_0|} \quad \textrm{and} \quad \kappa_r \left ( \frac{1}{\lambda_0}, rev P \right) = \kappa_r( \lambda_0, P). $$
\end{lemma}

\begin{proof}
We only prove the first claim. The second claim follows immediately from the first.

Let $x$ and $y$ be, respectively, a right and a left eigenvector of $P(\lambda)$ associated with $\lambda_0$. It is easy to see that these vectors are also a right and a left eigenvector of $rev P$ associated with $\frac{1}{\lambda_0}$. Notice that
\begin{align*}
\kappa_a\left (\frac{1}{\lambda_0}, rev P \right) & = \frac{(\sum_{i=0}^k \left | \frac{1}{\lambda_0}\right |^{k-i} \omega_i )\|x\|_2 \|y\|_2}{|y^*( rev P)' (\frac{1}{\lambda_0})x|}\\
&= \frac{(\sum_{i=0}^k \left | {\lambda_0}\right |^{i} \omega_i )\|x\|_2 \|y\|_2}{|\lambda_0|^k|y^*( rev P)' (\frac{1}{\lambda_0})x|}
=\frac{(\sum_{i=0}^k \left | {\lambda_0}\right |^{i} \omega_i )\|x\|_2 \|y\|_2}{|\lambda_0|^2|y^*\lambda_0^{k-2}( rev P)' (\frac{1}{\lambda_0})x|}\\
&=\frac{(\sum_{i=0}^k \left | {\lambda_0}\right |^{i} \omega_i )\|x\|_2 \|y\|_2}{|\lambda_0|^2|y^*P'(\lambda_0)x|}=\frac{\kappa_r(\lambda_0, P)}{|\lambda_0|},
\end{align*}
where the fourth equality follows from the facts that $revP(\lambda)= \lambda^k P(\frac{1}{\lambda}) $ and $P(\lambda_0)x=0$, and the first and fifth equalities follow from Theorem \ref{teorTis}. Thus, the claim follows.

%The second claim  follows from Theorem \ref{teorTis}.
\end{proof}

%%%%%%%%%%%%%%%%%%%%%%%%%%%
\subsection{Homogeneous eigenvalue condition numbers}

As pointed out in the last subsection, neither of the non-homogeneous condition numbers  is defined for infinite eigenvalues. Thus, these type of eigenvalues require a special treatment in the non-homogeneous setting. In this section we introduce two condition numbers that allow a unified approach to all eigenvalues, finite and infinite.  These condition numbers require  the matrix polynomial to be expressed  in homogeneous form (see \eqref{homform}). This is the reason why we refer to them as homogeneous eigenvalue condition numbers.

\begin{remark}\label{eig-P}
We recall that $(\alpha_0, \beta_0)\neq (0,0)$ is an eigenvalue of $P(\alpha, \beta)$ if and only if $\lambda_0:=\alpha_0/\beta_0$ is an eigenvalue of $P(\lambda)$, where $\lambda_0 =\infty$ if $\beta_0=0$.
\end{remark}

Each of the condition numbers presented in this subsection has been defined in the literature with a different approach. One of them is due to Stewart and Sun \cite{Stewart}, who define the eigenvalue condition number  in terms of the chordal distance between the exact and the perturbed eigenvalues. The other approach is due  to Dedieu and Tisseur \cite{Ded-first, DedTis2003} and makes use of the Implicit Function Theorem to construct a differential operator whose norm is defined to be an eigenvalue condition number. This idea was inspired  by Shub and Smale's work \cite{Shub}.  These two condition numbers  do not have a specific name in the literature.  We will refer to them as the Stewart-Sun  condition number and the  Dedieu-Tisseur condition number, respectively.

%%%%%%%%%%%%%%%
\subsubsection{Dedieu-Tisseur condition number}

The homogeneous eigenvalue condition number that we present in this section has been often used in recent literature on matrix polynomials as an alternative to the non-homogeneous Wilkinson-like condition number. See, for instance, \cite{hamarling, HigMacTis}. We do not include its explicit definition because it is much more involved than Definition \ref{rel-cond}. For the interested reader, the definition can be found in \cite{DedTis2003}. The next theorem provides an explicit formula for this condition number.

\begin{theorem}
	\label{teorhomcondnumb}{\rm\cite[Theorem 4.2]{DedTis2003}}
	Let $(\alpha_0,\beta_0)$ be a simple eigenvalue of the regular matrix polynomial $P(\alpha,\beta)=\sum_{i=0}^k\alpha^i\beta^{k-i}B_i$, and let  $y$ and $x$ be, respectively, a left and a right eigenvector of $P(\alpha, \beta)$  associated with $(\alpha_0, \beta_0)$. Then, the Dedieu-Tisseur condition number of $(\alpha_0,\beta_0)$ is given by
\begin{equation}\label{hom-form}
\kappa_{h}((\alpha_0,\beta_0),P)=\left (\sum_{i=0}^k|\alpha_0|^{2i}|\beta_0|^{2(k-i)} \omega_i^2 \right)^{1/2} \frac{\|y\|_2\|x\|_2}{|y^*(\overline{\beta_0}D_{\alpha} P(\alpha_0, \beta_0)-\overline{\alpha_0}D_{\beta} P(\alpha_0, \beta_0))x|},
\end{equation}
where $D_{a} \equiv \frac{\partial}{\partial a},$ that is, the partial derivative with respect to $a$ and $\omega_i$, $i=0:k$ are nonnegative weights that define how the perturbations of the coefficients $B_i$ are measured.
\end{theorem}

It is important to note that the expression for this eigenvalue condition number does not depend on the choice of representative for the eigenvalue  $(\alpha_0, \beta_0)$.

%%%%%%%%%%%%%%%%
\subsubsection{Stewart-Sun  condition number}
Here we introduce another homogeneous eigenvalue condition number.
Its definition is easy to convey and to interpret from a geometrical point of view.

We recall that every eigenvalue of a homogeneous matrix polynomial can be seen as a line in $\mathbb{C}^2$ passing through the origin. The condition number that we present here  uses the  ``chordal distance'' between lines in $\mathbb{C}^2$ to measure the distance between an eigenvalue and a perturbed eigenvalue. This distance is defined on the projective space $\mathbb{P}_1(\mathbb{C})$.

Before introducing the chordal distance, we recall the definition of angle between two lines.

\begin{definition}\label{cond-chord} Let $x$ and $y$ be two nonzero vectors in $\mathbb{C}^2$ and let $\langle x \rangle$ and $\langle y \rangle$ denote the lines passing through zero in the direction of $x$ and $y$, respectively. We define the \emph {angle between the two lines} $\langle x \rangle$ and $\langle y \rangle$   by
$$\theta(\langle x \rangle, \langle y \rangle) := arc cos \frac{|\langle x, y \rangle |}{\|x\|_2\|y\|_2}, \quad 0\leq\theta(\langle x \rangle, \langle y \rangle)\leq\pi/2,$$
where $\langle x, y \rangle$ denotes the standard Hermitian inner product, i.e., $\langle x,y\rangle=y^*x$.
\end{definition}

\begin{remark}\label{sintheta}
We note that $cos\; \theta(\langle x\rangle, \langle y \rangle)$ can be seen as the ratio between the length of the orthogonal projection (with respect to the standard inner product in $\mathbb{C}^2$) of the vector $x$ onto $y$ to the length of the vector $x$ itself, that is,
$$cos\; \theta(\langle x\rangle,\langle y\rangle ) = \frac{\|proj_y x\|_2}{\|x\|_2},$$
 since
$$proj_y x = \frac{\langle x, y \rangle y}{\|y\|_2^2} \quad \textrm{and} \quad
\frac{\|proj_y x\|_2}{\|x\|_2} = \frac{|\langle x, y \rangle|}{\|x\|_2 \|y\|_2}. $$
We also have
$$sin\; \theta(\langle x\rangle,\langle y\rangle )= \frac{\|x - proj_y x \|_2}{\|x\|_2} . $$
\end{remark}

The definition of chordal distance is given next.
\begin{definition}{\rm\cite[Chapter VI, Definition 1.20]{Stewart}}
	\label{def-chord}
Let $x$ and $y$ be two nonzero vectors in $\mathbb{C}^2$ and let $\langle x \rangle$ and $\langle y \rangle$ denote the lines passing through zero in the direction of $x$ and $y$, respectively.
The chordal distance between $\langle x \rangle$ and $\langle y \rangle$ is given by $$\chi(\langle x\rangle, \langle y \rangle):= sin(\theta(\langle x \rangle, \langle y \rangle)).$$
\end{definition}

Notice that
$\chi(\langle x \rangle, \langle y \rangle)\leq \theta(\langle x \rangle, \langle y \rangle).$ % \leq d_T(x,y),$$
Moreover, the chordal distance and the angle are identical asymptotically, that is, when $\theta(\langle x \rangle, \langle y \rangle)$ approaches $0$.

The chordal distance between two lines $\langle x \rangle$ and $\langle y \rangle $ in $\mathbb{C}^2$ can also be expressed in terms of the coordinates of the vectors $x$ and $y$ in the canonical basis for $\mathbb{C}^2$. Note that this expression does not depend on the representatives $x$ and $y$ of the lines.
\begin{lemma}\label{ds-formula}{\rm\cite[page 283]{Stewart}}
If $\langle \alpha, \beta \rangle$ and $\langle \gamma, \delta \rangle $ are two lines in $\mathbb{C}^2$, then
$$\chi(\langle \alpha, \beta \rangle, \langle \gamma, \delta \rangle)= \frac{|\alpha\delta - \beta \gamma|}{\sqrt{|\alpha|^2+|\beta|^2}\sqrt{|\gamma|^2+|\delta|^2}}.$$
\end{lemma}

\begin{remark}Notice that $0\leq\chi(\langle\alpha,\beta\rangle, \langle\gamma,\delta\rangle)\leq 1$ for all lines $\langle\alpha,\beta\rangle, \langle\gamma,\delta\rangle$ in $\mathbb{C}^2$. Since the line $\langle 1,0\rangle$ is identified with the eigenvalue $\infty$ in PEPs, we see that the chordal distance allows us to measure the distance from $\infty$ to any other eigenvalue very easily. Moreover, such distance is never larger than one.
\end{remark}

Next we introduce the homogeneous eigenvalue condition number in which the change in the eigenvalue is measured using  the chordal distance and that we baptize as the Stewart-Sun eigenvalue condition number. This condition number was implicitly introduced for matrix pencils in \cite[page 294]{Stewart}, although an explicit definition is not given in \cite{Stewart}. See also \cite[page 40]{bernahu} for an explicit definition of this condition number for matrix polynomials. 

Note that in Definition \ref{cord-cond} below, $(\alpha_0,+\Delta\alpha_0,\beta_0+\Delta\beta_0)$ is the unique simple eigenvalue of $(P+\Delta P)(\alpha,\beta)$ that approaches $(\alpha_0,\beta_0)$ when $\Delta P$ approaches zero.

\begin{definition}
\label{cord-cond}
Let $(\alpha_0, \beta_0)$ be a simple eigenvalue of a regular matrix polynomial
$P(\alpha, \beta)= \sum_{i=0}^k \alpha^i \beta^{k-i} B_i$ of grade $k$ and let $x$ be a right eigenvector of $P(\alpha, \beta)$ associated with $(\alpha_0, \beta_0)$. We define
\begin{align*}
& \kappa_{\theta}((\alpha_0,\beta_0),P):=\lim_{\epsilon\to 0}\sup \bigg\{\frac{\chi((\alpha_0,\beta_0),(\alpha_0+\Delta\alpha_0,\beta_0+\Delta\beta_0))}{\epsilon}:\\
& [P(\alpha_0+\Delta\alpha_0,\beta_0+\Delta\beta_0)+\Delta P(\alpha_0+\Delta\alpha_0,\beta_0+\Delta\beta_0)](x+\Delta x)=0,\\
& \|\Delta B_i\|_2\leq\epsilon\omega_i, i=0:k\bigg\},
\end{align*}
where  $\Delta P(\alpha, \beta)=\sum_{i=0}^k \alpha^i \beta^{k-i}\Delta B_i$ and $\omega_i$, $i=0:k$, are nonnegative weights that allow flexibility in how the perturbations of $P(\alpha, \beta)$ are measured.
\end{definition}

As far as we know, no explicit formula for the Stewart-Sun condition number is available in the literature. We provide such an expression next.

%%%%%%%%%%%%%%%%%%%%%%%%
%\section{Comparison of the Dedieu-Tisseur and Stewart-Sun condition numbers}

\begin{theorem}\label{form-kappatheta}
Let $(\alpha_0, \beta_0) \neq (0,0)$ be a simple eigenvalue of $P(\alpha, \beta)=\sum_{i=0}^k \alpha^i \allowbreak \beta^{k-i} B_i$, and let $x$ and $y$ be, respectively, a right and a left eigenvector of $P(\alpha, \beta)$ associated with $(\alpha_0, \beta_0)$. Then,
\begin{equation}
\kappa_{\theta}((\alpha_0,\beta_0),P)=\left (\sum_{i=0}^k|\alpha_0|^{i}|\beta_0|^{k-i} \omega_i \right) \frac{\|y\|_2\|x\|_2}{|y^*(\overline{\beta_0}D_{\alpha} P(\alpha_0, \beta_0)-\overline{\alpha_0}D_{\beta} P(\alpha_0, \beta_0))x|},
\end{equation}
where $D_{a} \equiv \frac{\partial}{\partial a},$ that is, the partial derivative with respect to $a$.\end{theorem}

\begin{proof}
Let  $(P+ \Delta P)(\alpha, \beta)$ be a small enough perturbation of $P(\alpha, \beta)$. Then, $(P+\Delta P)(\alpha, \beta)$ has a unique simple eigenvalue $(\tilde{\alpha_0}, \tilde{\beta_0})$, with associated eigenvector $x + \Delta x$, that approaches $(\alpha_0,\beta_0)$ when $\Delta P(\alpha,\beta)$ approaches zero. Since $\{[\alpha_0, \beta_0], [-\overline{\beta_0}, \overline{\alpha_0}]\}$ is an orthogonal basis for $\mathbb{C}^2$ (with respect to the standard inner product), we can choose a representative  for  $(\tilde{\alpha_0}, \tilde{\beta_0})$ of the form
$$[\alpha_0, \beta_0]+[\Delta \alpha_0, \Delta \beta_0],$$
where $\langle [\alpha_0, \beta_0],  [\Delta \alpha_0, \Delta \beta_0] \rangle = 0$ for any given representative $[\alpha_0,\beta_0]$ of $(\alpha_0,\beta_0)$. This implies that there exists a scalar $\frak{h}$ such that
\begin{equation}\label{kexpr}
[\Delta  \alpha_0, \Delta \beta_0]= \frak{h} [ -\overline{\beta_0}, \overline{\alpha_0}].
\end{equation}

Expanding for these representatives the left hand side of the constraint
$$[P(\alpha_0+\Delta \alpha_0, \beta_0+\Delta \beta_0) + \Delta P
(\alpha_0+\Delta \alpha_0, \beta_0+\Delta \beta_0)](x + \Delta x) =0$$
in the definition of $\kappa_{\theta}((\alpha_0, \beta_0), P)$ and keeping only the first order terms, we get
$$[D_{\alpha}P(\alpha_0, \beta_0) \Delta \alpha_0 + D_{\beta}P(\alpha_0, \beta_0) \Delta \beta_0]x + P(\alpha_0, \beta_0) \Delta x +\Delta P(\alpha_0, \beta_0) x = O(\epsilon^2).$$
If we multiply the previous equation by $y^*$ on the left, taking into account that $y$ is a left eigenvector of $P(\alpha, \beta)$ associated with $(\alpha_0, \beta_0)$, we get
\begin{equation}\label{eq-princ}
y^*[D_{\alpha}P(\alpha_0, \beta_0) \Delta \alpha_0 + D_{\beta}P(\alpha_0, \beta_0) \Delta \beta_0]x + y^*\Delta P(\alpha_0, \beta_0) x = O(\epsilon^2).
\end{equation}
%Note that, for every eigenvalue  $(\alpha_0+\Delta \alpha_0, \beta_0+\Delta \beta_0)$, we can choose a representative $[\Delta \alpha_0, \Delta \beta_0]$ so that the Hermitian inner product
%$$\langle [\alpha_0, \beta_0], [\Delta \alpha_0, \Delta \beta_0] \rangle =0.$$
%In other words, there exists a constant $\frak{h}$ such that $\Delta \alpha_0=\frak{h} \overline{\beta_0}$ and $\Delta \beta_0=-\frak{h} \overline{\alpha_0}$.
Using (\ref{kexpr}),  (\ref{eq-princ}) becomes
$$y^*[D_{\alpha}P(\alpha_0, \beta_0) \frak{h}\overline{\beta_0} - D_{\beta}P(\alpha_0, \beta_0) \frak{h} \overline{\alpha_0}]x - y^*\Delta P(\alpha_0, \beta_0) x = O(\epsilon^2).$$
Since $(\alpha_0, \beta_0) \neq 0$ is a simple eigenvalue, by \cite[Theorem 3.3]{DedTis2003},
$$y^*[D_{\alpha}P(\alpha_0, \beta_0) \overline{\beta_0} - D_{\beta}P(\alpha_0, \beta_0)  \overline{\alpha_0}]x \neq 0.$$
Therefore,
\begin{equation}\label{frakh}
\frak{h}= \frac{y^*\Delta P(\alpha_0, \beta_0) x}{y^*[\overline{\beta_0}D_{\alpha}P(\alpha_0, \beta_0)  - \overline{\alpha_0}D_{\beta}P(\alpha_0, \beta_0)  ]x } + O(\epsilon^2).
\end{equation}
On the other hand,
\begin{align}\label{chordalbound}
\frac{\chi((\alpha_0, \beta_0), (\alpha_0+\Delta \alpha_0, \beta_0+ \Delta  \beta_0))}{\epsilon} &= \frac{|\alpha_0 \Delta \beta_0 - \beta_0 \Delta\alpha_0|}{\epsilon \sqrt{|\alpha_0|^2+|\beta_0|^2} \sqrt{|\alpha_0+\Delta \alpha_0|^2 + |\beta_0 + \Delta \beta_0|^2}}\nonumber \\
&= \frac{|\frak{h}|}{\epsilon} \frac{\sqrt{|\alpha_0|^2+|\beta_0|^2}}{\sqrt{|\alpha_0+\Delta \alpha_0|^2 + |\beta_0 + \Delta \beta_0|^2}}.
\end{align}
Since, by (\ref{frakh}),
\begin{equation}\label{hbound}
\frac{|\frak{h}|}{\epsilon} \leq  \left (\sum_{i=0}^k|\alpha_0|^{i}|\beta_0|^{k-i} \omega_i \right) \frac{\|y\|_2\|x\|_2}{|y^*(\overline{\beta_0}D_{\alpha} P(\alpha_0, \beta_0)-\overline{\alpha_0}D_{\beta} P(\alpha_0, \beta_0))x|} + O(\epsilon)
\end{equation}
and $\Delta \alpha_0$ and $\Delta \beta_0$ approach zero as $\epsilon \to 0$, from (\ref{chordalbound}) and (\ref{hbound}) we get
$$ \kappa_{\theta}((\alpha_0, \beta_0), P) \leq   \left (\sum_{i=0}^k|\alpha_0|^{i}|\beta_0|^{k-i} \omega_i \right) \frac{\|y\|_2\|x\|_2}{|y^*(\overline{\beta_0}D_{\alpha} P(\alpha_0, \beta_0)-\overline{\alpha_0}D_{\beta} P(\alpha_0, \beta_0))x|} .$$

Now we need to show that this upper bound on the Stewart-Sun condition number can be attained. Let $$\Delta B_i = sgn(\overline{\alpha_0}^i)sgn(\overline{\beta_0}^{k-i}) \epsilon \omega_i \frac{yx^*}{\|x\|_2\|y\|_2},\quad i=0:k,$$
where $sgn(z)= \frac{z}{|z|}$ if $z\neq 0$ and $sgn(0)=0$.
Note that, with this definition of $\Delta B_i$, we have
$$ \|\Delta B_i \|_2 =  \epsilon \omega_i, \ i=0:k,\quad \text{and}\quad |y^* \Delta P(\alpha_0, \beta_0) x| =\epsilon \left( \sum_{i=0}^k |\alpha_0|^i |\beta_0|^{k-i} \omega_i \right) \|x\|_2 \|y\|_2.$$
Thus, the inequality in (\ref{hbound}) becomes an equality and the result follows.
\end{proof}

%%%%%%%%%%%%%%%%%%%%%%%%
\section{Comparisons of eigenvalue condition numbers of matrix polynomials}
\label{seccomp}
In this section we provide first a comparison between the Dedieu-Tisseur and Stewart-Sun homogeneous condition numbers and, as a consequence, we prove that these condition numbers are equivalent up to a moderate constant depending only on the degree of the polynomial. Then we compare the Stewart-Sun condition number with the non-homogeneous condition number in both its absolute and relative version; as a result, we see that these condition numbers can be very different in certain situations. A simple geometric interpretation of these differences is given in Section \ref{secgeom}. In the literature some comparisons can be found, as we will point out, but they provide inequalities among the condition numbers while our expressions are equalities.

%%%%%%%%%%%%%%%%%%%%
\subsection{Comparison of the Dedieu-Tisseur and Stewart-Sun condition numbers}
As mentioned earlier, the Dedieu-Tisseur and the Stewart-Sun homogeneous condition numbers are defined following a very different approach. So it is a natural question to determine how they are related. We start with a result known in the literature.

\begin{theorem}{\rm\cite[Section 7]{Ded-first}\cite[Corollary 2.6]{bernahu}}
Let $(\alpha_0, \beta_0) \neq (0,0)$ be a simple eigenvalue of a regular matrix polynomial
$P(\alpha, \beta)= \sum_{i=0}^k \alpha^i \beta^{k-i} B_i$ of grade $k$. Assuming that the same weights $\omega_i$ are considered for both condition numbers, we have
$$ \kappa_{\theta}((\alpha_0, \beta_0), P)  \leq C \kappa_h((\alpha_0, \beta_0), P),$$
for some constant $C$.
\end{theorem}

To provide an exact relationship between the Dedieu-Tisseur and the Stewart-Sun homogeneous condition numbers, we simply use the  explicit formulas given for them  in Theorems \ref{teorhomcondnumb} and \ref{form-kappatheta}, respectively.
\begin{theorem}\label{DT-SS}
Let $(\alpha_0, \beta_0) \neq (0,0)$ be a simple eigenvalue of $P(\alpha, \beta)=\sum_{i=0}^k \alpha^i \allowbreak \beta^{k-i}B_i$. Then,
$$\kappa_h((\alpha_0, \beta_0), P) = \frac{\left(\sum_{i=0}^k|\alpha_0|^{2i}|\beta_0|^{2(k-i)} \omega_i^2 \right)^{1/2} }{\sum_{i=0}^k |\alpha_0|^{i}|\beta_0|^{k-i} \omega_i} \ \kappa_{\theta}((\alpha_0, \beta_0), P).$$
\end{theorem}

The following result is an immediate consequence of Theorem \ref{DT-SS} and shows that, for moderate $k$, both homogeneous condition numbers are essentially the same.

\begin{corollary}
\label{coro-comp}
Let $(\alpha_0, \beta_0) \neq (0,0)$ be a simple eigenvalue of $P(\alpha, \beta)=\sum_{i=0}^k \alpha^i \beta^{k-i}B_i$. Then,
$$\frac{1}{\sqrt{k+1}}\leq \frac{\kappa_{h}((\alpha_0, \beta_0), P)}{\kappa_{\theta}((\alpha_0, \beta_0), P)} \leq 1.$$
\end{corollary}

\begin{proof}
Let $v:=[|\beta_0|^k \omega_0, |\alpha_0||\beta_0|^{k-1} \omega_1, \ldots, |\alpha_0|^k \omega_k]^T\in\mathbb{R}^{k+1}$. From Theorem \ref{DT-SS}, we have
\begin{align*}
\frac{\kappa_{h}((\alpha_0, \beta_0), P)}{\kappa_{\theta}((\alpha_0, \beta_0), P)} = \frac{\|v\|_2}{\|v\|_1},
\end{align*}
where $\|.\|_1$ denotes the vector 1-norm \cite{Stewart}.
The result follows taking into account the fact that
$1/\sqrt{k+1}\leq\|v\|_2/\|v\|_1\leq 1.$
\end{proof}

Since the Dedieu-Tisseur and the Stewart-Sun condition numbers are equivalent and the definition of the Stewart-Sun condition number is much simpler and intuitive, we do not see any advantage in using the Dedieu-Tisseur condition number. Therefore, in the next subsection, we focus on comparing the Stewart-Sun condition number with the non-homogeneous condition numbers. The corresponding comparisons with the Dedieu-Tisseur condition number follow immediately from Theorem \ref{DT-SS} and Corollary \ref{coro-comp}.

\subsection{Comparison of the homogeneous and non-homogeneous eigenvalue condition numbers}

As we mentioned in Section \ref{sec2}, the main drawback of the non-homogeneous condition numbers is that they do not allow a unified treatment of all the eigenvalues of a matrix polynomial since these condition numbers are not defined for the infinite eigenvalues. Thus, some researchers prefer to use a homogeneous eigenvalue condition number instead, although in most applications the non-homogeneous eigenvalues $\lambda$ are the relevant quantities. In this section, we give an algebraic relationship between the Stewart-Sun condition number and the non-homogeneous (absolute and relative) eigenvalue condition number. We emphasize again that, by Theorem \ref{DT-SS} and Corollary \ref{coro-comp}, this relation also provides us with a relation between the Dedieu-Tisseur condition number and the non-homogeneous condition numbers.

We start with the only result we have found in the literature on this topic.

\begin{theorem}{\rm\cite[Corollary 2.7]{bernahu}}
Let $(\alpha_0, \beta_0)$ with $\alpha_0\neq 0$ and $\beta_0\neq 0$ be a simple eigenvalue of a regular matrix polynomial
$P(\alpha, \beta)= \sum_{i=0}^k \alpha^i \beta^{k-i} B_i$ of grade $k$ and let $\lambda_0:=\frac{\alpha_0}{\beta_0}$. Assuming that the same weights $\omega_i$ are considered for both condition numbers, there exists a constant $C$ such that
$$ \kappa_r(\lambda_0, P) \leq C\ \frac{1+|\lambda_0|^2}{|\lambda_0|}\ \kappa_h((\alpha_0, \beta_0), P).$$
\end{theorem}

The next theorem is the main result in this section and provides exact relationships between the Stewart-Sun condition number and the non-homogeneous condition numbers. Note that in Theorem \ref{ThKrtheta} an eigenvalue condition number of the reversal of a matrix polynomial is used, more precisely, $\kappa_a((1/\lambda_0),rev P)$. For $\lambda_0=\infty$, this turns into $\kappa_a(0, rev P)$ which can be interpreted as a non-homogeneous absolute condition number of $\lambda_0=\infty$.

\begin{theorem} \label{ThKrtheta}
Let $(\alpha_0, \beta_0)\neq (0,0)$ be a simple eigenvalue of a regular matrix polynomial
$P(\alpha, \beta)= \sum_{i=0}^k \alpha^i \beta^{k-i} B_i$ of grade $k$ and let $\lambda_0:=\frac{\alpha_0}{\beta_0}$, where $\lambda_0=\infty$ if $\beta_0=0$.
Assume that the same weights $\omega_i$ are considered in the definition of all the condition numbers appearing below, i.e., $\|\Delta B_i\|_2\leq\epsilon\omega_i, i=0:k$, in all of them. Then,
\begin{enumerate}
\item[(i)] if $\beta_0 \neq 0$,
\begin{equation}\label{abs-theta}
 \kappa_{\theta}((\alpha_0, \beta_0), P) = \kappa_a(\lambda_0, P)\ \frac{1}{1+|\lambda_0|^2};
 \end{equation}
\item[(ii)] if $\alpha_0 \neq 0$,
\begin{equation}\label{abs-theta-rev}
 \kappa_{\theta}((\alpha_0, \beta_0), P) = \kappa_a \left (\frac{1}{\lambda_0}, rev P\right)\ \frac{1}{1+\left |\frac{1}{\lambda_0}\right|^2} ;
 \end{equation}
 \item[(iii)] if $\alpha_0\neq 0$ and $\beta_0 \neq 0$,
 \begin{equation}\label{rel-theta}
 \kappa_{\theta}((\alpha_0, \beta_0), P) = \kappa_r(\lambda_0, P)\ \frac{|\lambda_0|}{1+|\lambda_0|^2} .
\end{equation} \end{enumerate}

\end{theorem}

\begin{proof}
Assume first that $\beta_0\neq 0$, which implies that $\lambda_0$ is finite.  Let $(\tilde{\alpha_0}, \tilde{\beta_0}):= (\alpha_0+ \Delta \alpha_0, \beta_0+\Delta \beta_0)$ be a perturbation of $(\alpha_0, \beta_0)$ small enough so that $\beta_0+\Delta \beta_0 \neq 0$, and let $\lambda_0+ \Delta \lambda_0:=\frac{\tilde{\alpha_0}}{\tilde{\beta_0}}$. Note that
\begin{equation}\label{deltala-be-al}
\Delta \lambda_0 = \frac{\tilde{\alpha_0} \beta_0 - \alpha_0 \tilde{\beta_0}}{\beta_0 \tilde{\beta_0}}=\frac{\beta_0\Delta \alpha_0 - \alpha_0\Delta \beta_0}{\beta_0(\beta_0+\Delta \beta_0)} .
\end{equation}
Then, we have, by Lemma \ref{ds-formula},
\begin{align}
\chi((\alpha_0, \beta_0), (\tilde{\alpha_0}, \tilde{\beta_0}))
& =  \frac{| \alpha_0 \Delta{\beta_0} - \beta_0 \Delta{\alpha_0}|}{|\beta_0| \sqrt{1+|\lambda_0|^2}|\beta_0+\Delta{\beta_0}|  \sqrt{1+|\lambda_0 +\Delta \lambda_0|^2}} \nonumber\\
&= \frac{| \Delta \lambda_0|}{\sqrt{1+|\lambda_0|^2} \sqrt{1+|\lambda_0+\Delta \lambda_0|^2}}\label{eq1}
\end{align}
where the second equality follows from (\ref{deltala-be-al}).
 Then, as $\epsilon \to 0$ in the definition of the condition numbers (which implies that $|\Delta\alpha_0|$ and $|\Delta \beta_0|$ approach $ 0$ as well, using a continuity argument), we have $\Delta \lambda_0 \to 0.$
Thus, bearing in mind Definitions \ref{rel-cond} and \ref{cord-cond}, (\ref{abs-theta}) follows from (\ref{eq1}).

Assume now that $\alpha_0\neq 0$ and $\beta_0 \neq 0$, that is, $\lambda_0 \neq 0$ and $\lambda_0$ is not  an infinite eigenvalue. Then, from (\ref{eq1}), we get
$$\chi((\alpha_0,\beta_0),(\tilde\alpha_0,\tilde\beta_0))= \, \frac{|\Delta \lambda_0|}{|\lambda_0|} \frac{|\lambda_0|}{\sqrt{1+|\lambda_0|^2} \sqrt{1+|\lambda_0+\Delta \lambda_0|^2}}.$$
This implies (\ref{rel-theta}). By Lemma \ref{kappa-a-r}, (\ref{abs-theta-rev}) follows from \eqref{rel-theta} for finite, nonzero eigenvalues.
It only remains to prove \eqref{abs-theta-rev} for $\lambda_0=\infty$, which corresponds to $(\alpha_0,\beta_0)=(1,0)$. This follows from Theorem \ref{teorTis} applied to $rev P$ and Theorem \ref{form-kappatheta}, which yield:
$$\kappa_a(0, rev P) = \frac{\omega_k \|x\|_2 \|y\|_2}{|y^*B_{k-1}x|}=\kappa_{\theta}((1,0), P).$$
\end{proof}

From Theorem \ref{ThKrtheta}, we immediately get
$$\kappa_{\theta}((\alpha_0,\beta_0),P)\leq\kappa_r(\lambda_0,P) \quad \text{for} \ 0<|\lambda_0|<\infty,$$
as a consequence of \eqref{rel-theta}, and
$$\kappa_{\theta}((\alpha_0,\beta_0),P)\leq\kappa_a(\lambda_0,P)\quad \text{for}\ 0\leq|\lambda_0|<\infty,  $$
as a consequence of \eqref{abs-theta}. In addition, Theorem \ref{ThKrtheta} guarantees that there exist values of $\lambda_0$ for which $\kappa_\theta((\alpha_0,\beta_0),P)$ and $\kappa_r(\lambda_0,P)$ are very different and also values for which they are very similar. The same happens for $\kappa_{\theta}((\alpha_0,\beta_0),P)$ and $\kappa_a(\lambda_0,P)$. In the rest of this subsection we explore some of these scenarios.

The next result follows directly from Theorem \ref{ThKrtheta} and says that for small $|\lambda_0|$, $\kappa_{\theta}((\alpha_0, \beta_0),P)$ is essentially $\kappa_a(\lambda_0, P)$, while for  $|\lambda_0| \geq 1$, $\kappa_{\theta}((\alpha_0, \beta_0),P)$ is essentially $\kappa_a(1/\lambda_0, rev P)$.

\begin{corollary}\label{interp1}
Let $(\alpha_0, \beta_0)\neq (0,0)$ be a simple eigenvalue of a regular matrix polynomial
$P(\alpha, \beta)= \sum_{i=0}^k \alpha^i \beta^{k-i} B_i$ of grade $k$ and let $\lambda_0:=\frac{\alpha_0}{\beta_0}$, where $\lambda_0=\infty$ if $\beta_0=0$.
Assume that the same weights $\omega_i$ are considered in the definition of the condition numbers that appear below. Then
\begin{itemize}
\item[(i)]If $0 \leq |\lambda_0| \leq 1$, we have
\begin{equation}\label{less1}
\frac{\kappa_a(\lambda_0, P)}{2} \leq \kappa_{\theta}((\alpha_0, \beta_0), P)  \leq \kappa_a(\lambda_0, P).
\end{equation}
\item[(ii)] If  $|\lambda_0| \geq  1$, we have
\begin{equation}\label{more1}
\frac{\kappa_a\left (\frac{1}{\lambda_0}, rev P \right)}{2}   \leq \kappa_{\theta}\left ((\alpha_0, \beta_0), P \right)  \leq  \kappa_a \left (\frac{1}{\lambda_0}, rev P \right).
\end{equation}
\end{itemize}
\end{corollary}

From Corollary \ref{interp1}, we can informally state that for $0\leq|\lambda_0|\leq 1$, $\kappa_{\theta}((\alpha_0,\beta_0),P)$ behaves more as an absolute than as a relative eigenvalue condition number of $P$. The same happens for $|\lambda_0|\geq 1$ but with respect to $rev P$ and $1/\lambda_0$.

Taking into account  Lemma \ref{kappa-a-r} and the fact that $\kappa_a(\lambda_0, P) = \kappa_r(\lambda_0,P)\cdot |\lambda_0|$, the observations in (\ref{less1}) and (\ref{more1})  lead to the following conclusions.

\begin{corollary}\label{interp2}
Let $(\alpha_0, \beta_0)\neq (0,0)$ be a simple eigenvalue of a regular matrix polynomial
$P(\alpha, \beta)= \sum_{i=0}^k \alpha^i \beta^{k-i} B_i$ of grade $k$ and let $\lambda_0:=\frac{\alpha_0}{\beta_0}$, where $\lambda_0=\infty$ if $\beta_0=0$.
Assume that the same weights $\omega_i$ are considered in the definition of the condition numbers appearing below. Then,
\begin{itemize}
\item[(i)]  If $0< |\lambda_0| \leq 1$, \text{we have}
$$\kappa_r(\lambda_0, P)\frac{|\lambda_0|}{2}\leq \kappa_{\theta}((\alpha_0, \beta_0), P)  \leq \kappa_r(\lambda_0, P)|\lambda_0|,$$
\item[(ii)] If  $1 \leq |\lambda_0| < \infty$, \text{we have}
\begin{equation*}\label{more2}
\kappa_r(\lambda_0, P)\frac{1}{2|\lambda_0|}\leq \kappa_{\theta}((\alpha_0, \beta_0), P)  \leq \kappa_r(\lambda_0, P)\frac{1}{|\lambda_0|}.
\end{equation*}
\end{itemize}
\end{corollary}

\begin{remark}\label{kappainterp}
Using the results in this subsection, we can give a description (in function of $|\lambda_0|$)  of the behavior of the Stewart-Sun condition number in terms of the non-homogeneous condition numbers. For this purpose, we use $\kappa_a(\lambda_0,P)=\kappa_r(\lambda_0,P)|\lambda_0|$ to obtain the following relations:
\begin{itemize}
	\item[(i)] From (\ref{abs-theta}), we get that,  if $|\lambda_0|=1$, then
	$$\kappa_{\theta}((\alpha_0, \beta_0), P) = \frac{\kappa_{a}(\lambda_0, P)}{2} = \frac{\kappa_r(\lambda_0, P)}{2}.$$
	\item[(ii)] From (\ref{less1}), we get, as $|\lambda_0|$ approaches 0,
	$$\kappa_{\theta}((\alpha_0, \beta_0), P) \approx \kappa_{a}(\lambda_0, P) \ll \kappa_r(\lambda_0, P).$$
	\item[(iii)]  From (\ref{more1}), Lemma \ref{kappa-a-r}, and the fact that  $\kappa_a(1/\lambda_0, rev P) = \kappa_r(1/\lambda_0,rev P)\cdot \left |\frac{1}{\lambda_0}\right|$, we get, as $|\lambda_0|$ approaches $\infty$,
	$$\kappa_{\theta}((\alpha_0, \beta_0), P) \approx \kappa_{a}\left  (\frac{1}{\lambda_0}, rev P  \right) \ll \kappa_r\left (\frac{1}{\lambda_0}, rev P\right )= \kappa_{r}(\lambda_0, P) \ll \kappa_a(\lambda_0, P).$$
\end{itemize}
\end{remark}

%%%%%%%%%%%%%%%%%%%%%
\section{A geometric interpretation of the relationship between homogeneous and non-homogeneous condition numbers}
\label{secgeom}

In Theorem \ref{ThKrtheta}, we provided an exact relationship between the Stewart-Sun homogeneous condition number and the non-homogeneous condition numbers. This relationship involves the factors  $\frac{1}{1+|\lambda_0|^2}$ and $\frac{1}{1+\left |\frac{1}{\lambda_0}\right|^2}$, when the Stewart-Sun condition number is compared with the absolute non-homogeneous condition number, and involves the factor $\frac{|\lambda_0|}{1+|\lambda_0|^2}$,  when the Stewart-Sun {condition} number is compared with the relative non-homogeneous condition number. In this section we give a geometric interpretation of these factors, which leads to a natural understanding of the situations discussed in Remark \ref{kappainterp} where the homogeneous and non-homogeneous condition numbers can be very different. The reader should bear in mind, once again, that in most applications of PEPs, the quantities of interest are the non-homogeneous eigenvalues, which can be accurately computed with the current algorithms only if the non-homogeneous condition numbers are moderate.

\subsection{Geometric interpretation in terms of the chordal distance}
The factors $\frac{1}{1+|\lambda_0|^2}$, $\frac{1}{1+\left | \frac{1}{\lambda_0}\right |^2}$ and $\frac{|\lambda_0|}{1+|\lambda_0|^2}$ appearing in (\ref{abs-theta}),  (\ref{abs-theta-rev}) and (\ref{rel-theta}), respectively, can be interpreted in terms of the chordal distance as we show in the next theorem. We do not include a proof of this result since it can be immediately obtained from the  definition of chordal distance {(Definition \ref{def-chord})} and Remark \ref{sintheta}. Note that $\chi((\alpha_0, \beta_0), (1,0))$ can be seen as the chordal distance from ``$\lambda_0=\frac{\alpha_0}{\beta_0}$ to $\infty$'', while  $\chi((\alpha_0, \beta_0), (0,1))$ can be seen as the chordal distance from ``$\lambda_0=\frac{\alpha_0}{\beta_0}$ to 0''.

\begin{proposition}\label{geometric1}
Let $(\alpha_0, \beta_0)\neq (0,0)$  and let $\lambda_0:=\frac{\alpha_0}{\beta_0}$, where $\lambda_0=\infty$ if $\beta_0=0$.  Let $
\theta$ denote the angle between {$(\alpha_0, \beta_0 )$ and $(1,0)$.}
Then,

\begin{itemize}
\item[(i)] If $\beta_0 \neq 0$, then
$$\frac{1}{1+|\lambda_0|^2}=\chi((\alpha_0, \beta_0), (1,0))^2=\sin^2(\theta).$$
\item[(ii)] If $\alpha_0\neq 0$, then
$$\frac{1}{1+\left |\frac{1}{\lambda_0}\right|^2}= \chi((\alpha_0, \beta_0), (0,1))^2 = \cos^2(\theta).$$
\item[(iii)] If $\alpha_0\neq 0$ and $\beta_0 \neq 0$, then
$$\frac{|\lambda_0|}{1+|\lambda_0|^2} = \chi((\alpha_0, \beta_0), (1,0))\ \chi((\alpha_0, \beta_0), (0,1))= \sin(\theta) \cos(\theta)  .$$
\end{itemize}
\end{proposition}

\begin{remark}
We notice that the conditions $0 \leq |\lambda_0| \leq 1$ and $|\lambda_0| \geq 1$ used in Corollaries \ref{interp1} and \ref{interp2} and in Remark \ref{kappainterp} are equivalent, respectively, to $\chi((\alpha_0, \beta_0), \allowbreak (1,0)) \geq 1/\sqrt{2}$ and $\chi((\alpha_0, \beta_0), (0,1)) \geq 1/\sqrt{2}$, or, in other words, to $sin(\theta) \geq 1/\sqrt{2}$ and $cos(\theta) \geq 1/ \sqrt{2}$.
\end{remark}

Combining Proposition \ref{geometric1} (iii) with Theorem \ref{ThKrtheta} (iii), we see that either when the angle between the lines $(\alpha_0,\beta_0)$ and $(1,0)$ is very small or when the angle between the lines $(\alpha_0, \beta_0)$ and $(0,1)$ is very small, $\kappa_{\theta}((\alpha_0,\beta_0),P)\ll\kappa_{r}(\lambda_0,P)$, i.e., even in the case $\kappa_{\theta}((\alpha_0,\beta_0),P)$ is moderate and the line $(\alpha_0,\beta_0)$ changes very little under perturbations, the quotient $\lambda_0=\alpha_0/\beta_0$ can change a lot in a relative sense. This is immediately understood geometrically in $\mathbb{R}^2$. The combination of the remaining parts of Theorem \ref{ThKrtheta} and Proposition \ref{geometric1} lead to analogous discussions. In the next section, we make an analysis of these facts from another perspective.
%%%%%%%%%%%%%%%%%%%%

\subsection{Geometric interpretation in terms of the condition number of the  cotangent function}

Let $\alpha_0, \beta_0 \in \mathbb{C}$ with $\beta_0\neq 0$, let $\lambda_0:=\frac{\alpha_0}{\beta_0}$, and let $\theta:=\theta((\alpha_0, \beta_0), (1,0))$, that is, let $\theta$ denote the angle between the lines $(\alpha_0, \beta_0)$ and $(1,0)$. From Proposition \ref{geometric1},
$$cos \;\theta=  \frac{|\lambda_0|}{\sqrt{1+|\lambda_0|^2}}, \quad sin \; \theta =\frac{1}{\sqrt{1+|\lambda_0|^2}}.$$
Thus,
$$|\lambda_0| = \frac{|\alpha_0|}{|\beta_0|} = cotan\; \theta \quad \textrm{and} \quad  \frac{1}{|\lambda_0|} = tan \; \theta.$$

Note that this is also the standard definition of the cotangent and tangent functions in the first quadrant of $\mathbb{R}^2$. The cotangent function is differentiable in $(0, \pi/2)$. Thus, the absolute condition number\footnote{{When we refer to the \textit{absolute} condition number of a function of $\theta$, we mean that we are measuring the changes both in the function and in $\theta$ in an absolute sense.}} of this function is
\begin{equation}\label{kappaact}
\kappa_{a,ct}(\theta):=|cotan'(\theta)| = |1+cotan^2(\theta)| = 1 + |\lambda_0|^2,
\end{equation}
{which is huge when $\theta$ approaches zero.}
Moreover, the relative-absolute condition number\footnote{{When we mention the \textit{relative-absolute} condition number of a function of $\theta$, we consider that we are measuring the change in the function in a relative sense and the one in $\theta$ in an absolute sense.}} of the cotangent function is given by
\begin{equation}\label{kapparct}
\kappa_{r,ct}(\theta):= \frac{|cotan'(\theta)|}{|cotan(\theta)|}= \frac{|1+cotan^2(\theta)|}{|cotan(\theta)|}=\frac{1+|\lambda_0|^2}{|\lambda_0|},
\end{equation}
which is huge when $\theta$ approaches either zero or $\pi/2$.

The tangent function is  also differentiable in $(0,\pi/2)$ and the absolute condition number of this function is
\begin{equation}\label{kappaat}
\kappa_{a,t}(\theta):= |tan'(\theta)|=|1+tan^2(\theta)| = 1 + \left | \frac{1}{\lambda_0} \right|^2.
\end{equation}

From (\ref{abs-theta}), (\ref{abs-theta-rev}),  (\ref{rel-theta}), (\ref{kappaact}), (\ref{kapparct}), and (\ref{kappaat}), we obtain the following result.

\begin{theorem}\label{cotangent}
Let $(\alpha_0, \beta_0)\neq (0,0)$ be a simple eigenvalue of a regular matrix polynomial
$P(\alpha, \beta)= \sum_{i=0}^k \alpha^i \beta^{k-i} B_i$ of grade $k$ and let $\lambda_0:=\frac{\alpha_0}{\beta_0}$, where $\lambda_0=\infty$ if $\beta_0=0$. Let $\theta:=\theta((\alpha_0, \beta_0), (1,0)) $.
Assume that the same weights are considered in the {definitions of all the condition numbers appearing below}. Then,
\begin{itemize}
	\item [(i)] If $\beta_0\neq 0$, then
	\begin{equation}\label{abs-theta-nh}
	\kappa_a(\lambda_0, P) = \kappa_{\theta}((\alpha_0, \beta_0), P)\ \kappa_{a,ct}(\theta).
	\end{equation}
	\item[(ii)] If $\alpha_0\neq 0$, then
	\begin{equation*}\label{rev-theta-nh}
	\kappa_a \left (\frac{1}{\lambda_0}, rev P\right) =\kappa_{\theta}((\alpha_0, \beta_0), P)\ \kappa_{a,t}(\theta).
	\end{equation*}
	\item[(iii)] If $\alpha_0\neq 0$ and $\beta_0\neq 0$, then
	\begin{equation}\label{rel-theta-nh}
	\kappa_r(\lambda_0, P) = \kappa_{\theta}((\alpha_0, \beta_0), P)\ \kappa_{r,ct}(\theta).
	\end{equation}
\end{itemize} 
\end{theorem}

{Since for lines $(\alpha_0,\beta_0)$ and $(\tilde\alpha_0,\tilde\beta_0)$ very close to each other,  $\chi((\alpha_0,\beta_0),(\tilde{\alpha_0},\tilde{\beta_0}))\approx\theta((\alpha_0,\beta_0),(\tilde{\alpha_0},\tilde{\beta_0}))=|\theta((\alpha_0,\beta_0),(1,0))-\theta((\tilde{\alpha_0},\tilde{\beta_0}),(1,0))|$, equations \eqref{abs-theta-nh} and \eqref{rel-theta-nh} express the non-homogeneous condition numbers of $\lambda_0$ as a combination of two effects: the change of the homogeneous eigenvalue measured by $\theta((\alpha_0,\beta_0),(\tilde{\alpha_0},\tilde{\beta_0}))$ as a consequence of perturbations in the coefficients of $P(\alpha,\beta)$ and the alteration that this change produces in $|cotan(\theta)|$, which depends only on the properties of $cotan(\theta)$ and not on $P(\lambda)$. In fact, with this idea in mind, Theorem \ref{cotangent} can also be obtained directly from the definitions of the involved condition numbers.}

We notice that the expressions in (\ref{abs-theta-nh}) and (\ref{rel-theta-nh}) can be interpreted as follows: Given a matrix polynomial $P(\lambda)$, the usual way to solve the polynomial eigenvalue problem is to use a linearization {$L(\lambda)=\lambda L_1-L_0$} of $P(\lambda)$. A standard algorithm to solve the generalized eigenvalue problem associated with $L(\lambda)$ is the QZ algorithm. This algorithm computes first the generalized Schur decomposition of $L_1$ and $L_0$, that is, these matrix coefficients are  factorized in the form $L_1=QSZ^*$ and $L_0=QTZ^*$, where $Q$ and $Z$ are unitary matrices and  $S$ and $T$ are upper-triangular matrices. The pairs {$(T_{ii}, S_{ii})$}, where $S_{ii}$ and $T_{ii}$ denote the main diagonal entries of  $S$ and $T$ in position $(i,i)$, respectively,  are the ``homogeneous'' eigenvalues of $L(\lambda)$ (and, therefore, of $P(\lambda)$). In order to obtain the {non-homogeneous} eigenvalues of $P(\lambda)$, one more step is necessary, namely, to divide {$T_{ii}/S_{ii}$}. The expressions in (\ref{abs-theta-nh}) and (\ref{rel-theta-nh}) say that, even if {$\kappa_{\theta}((T_{ii},S_{ii}),P)$ is moderate and} the pair {$(T_{ii}, S_{ii})$} is ``{accurately} computed'',  the  quotient {$\lambda_i:=T_{ii}/S_{ii}$} may be ``{inaccurately computed"} when {$S_{ii}$} is very close to zero (that is, when $|\lambda_i|$ is very large) or when {$T_{ii}$} is close to zero (that is, when $|\lambda_i|$ is close to zero) {since $|\lambda_i|$ will have a huge non-homogeneous condition number. More precisely, for} the large eigenvalues, both {$\kappa_{a}(\lambda_0, P)$} and $\kappa_{r}(\lambda_0, P)$ will be much larger than $\kappa_{\theta}((\alpha_0, \beta_0), P)$, and for the small eigenvalues, $\kappa_{r}(\lambda_0, P)$ will be much larger than $\kappa_{\theta}((\alpha_0, \beta_0), P)$. This observation brings up the question of {the} computability of small and large eigenvalues.

\subsection{Computability of small and large eigenvalues {of matrix polynomials}}

By Remark \ref{kappainterp},  if $|\lambda_0|$ is very large, we have
\begin{equation}
 \kappa_{\theta}((\lambda_0, 1), P) \ll \kappa_r(\lambda_0, P) \ll \kappa_a(\lambda_0, P),
 \end{equation}
and, if  $|\lambda_0|$ is very close to 0, then
\begin{equation}
 \kappa_{\theta}((\lambda_0, 1), P)\approx \kappa_{a}(\lambda_0, P) \ll \kappa_r(\lambda_0, P).
 \end{equation}
The question that we study in this section is whether or not the {absolute non-homogeneous} condition number $\kappa_a(\lambda_0, P) $ is always very large when $|\lambda_0|$ is very large, or the {relative non-homogeneous} condition number $\kappa_r(\lambda_0, P)$ is always very large when either $|\lambda_0|$ is very close to 0 or is very large. If this was the case, then these types of {non-homogeneous} eigenvalues {would be always very ill-conditioned and would be computed with such huge errors by the available algorithms that it could be simply said that they are not  computable.}

We focus our answer to this question on the behavior of the eigenvalues of pencils (that is, we focus on matrix polynomials  $P(\lambda)$  of grade 1). This is reasonable since, when computing the eigenvalues of a matrix polynomial, the most common approach is to use a linearization. In order to keep in mind that we are not working with general matrix polynomials, we will use the notation $L(\lambda)$ to denote a pencil instead of $P(\lambda)$. {Moreover, we will focus on eigenvalue condition numbers with weights $\omega_i$ corresponding to the backward errors of current algorithms for generalized eigenvalue problems.%Our conclusion is that, in some occasions, the condition number $\kappa_r(\lambda_0, P)$ can be small, even if $\lambda_0$ is close to 0 or infinity, while in others it can be very large. %as the following example shows.

%Note that, if    $|\lambda_0|$ is very large, then $\kappa_r(\lambda_0, L)$ (and, therefore $\kappa_a(\lambda_0, L)$) can be very large unless  $\kappa_r(\lambda_0, L)\approx \frac{|\lambda_0|}{1+|\lambda_0|^2} $. A similar argument holds for  $|\lambda_0| \approx 0$, regarding only $\kappa_r(\lambda_0, L)$.

%In order to explore if the latter is possible, we give the following technical lemma.
%\begin{lemma}\label{lower-bound}
Let $L(\lambda):=\lambda B_1+B_0$ be a {regular} pencil and let $\lambda_0$ be a finite,  simple eigenvalue of $L(\lambda)$.
%$$\kappa_{\theta}((\lambda_0,1),L) \geq \frac{\min\{\omega_0,\omega_1\}}{\max\{\|B_0\|_2,\|B_1\|_2\}},$$
%where $\omega_0, \omega_1$ are the weights used in the computation of $\kappa_{\theta}((\lambda_0, 1), P)$.
%\end{lemma}
%\begin{proof}
Notice that,  by Theorem \ref{form-kappatheta},
\begin{align}\label{inequality}
\kappa_{\theta}((\lambda_0,1),L)&=\frac{(\omega_0+|\lambda_0|\omega_1)\|y\|_2\|x\|_2}{|y^*(B_1-\overline{\lambda_0}B_0)x|}
\geq\frac{(\omega_0+|\lambda_0|\omega_1)}{\|B_1\|_2+|\lambda_0|\|B_0\|_2}.%\geq\frac{\min\{\omega_0,\omega_1\}}{\max\{\|B_0\|_2,\|B_1\|_2\}}.
\end{align}

%and
%\begin{align*}\kappa_{\theta}((\lambda_0,1),L)&=\frac{(\omega_0+|\lambda_0|\omega_1)\|y\|_2\|x\|_2}{(1+|\lambda_0|^2) |y^* B_1 x|}
%\geq\frac{(\omega_0+|\lambda_0|\omega_1)}{(1+|\lambda_0|^2)\|B_1\|_2}.%\geq\frac{\min\{\omega_0,\omega_1\}}{\max\{\|B_0\|_2,\|B_1\|_2\}}.
%\end{align*}

%\end{proof}

The following result is an immediate consequence of the previous inequality and shows some important cases in which  the  eigenvalues of a pencil $L(\lambda)$ with large or small modulus are not computable.

 % and the weights $\omega_0$ and $\omega_1$ are chosen as in Remark \ref{weights}.

\begin{proposition}\label{comput}
Let $L(\lambda)=\lambda B_1+B_0$ be a {regular} pencil and let $\lambda_0$ be a finite, simple eigenvalue of $P(\lambda)$. Let $\omega_1, \omega_0$ be the weights used in the definition of the {non-homogeneous} condition number of $\lambda_0$.  Then,
$$\kappa_{a}(\lambda_0, L) \geq 1+|\lambda_0|^2 \quad \textrm{and} \quad  \kappa_{r}(\lambda_0, L) \geq \frac{1+|\lambda_0|^2}{|\lambda_0|}, $$
(where the second inequality holds only if $\lambda_0 \neq 0$) if any of the following {conditions} holds:
\begin{enumerate}
\item $\omega_0= \omega_1=\max\{\|B_1\|_2, \|B_0\|_2\}$;
\item $\omega_i =\|B_i\|_2$ for $i=0,1$, $|\lambda_0| < 1$ and $\|B_1\|_2 \leq \|B_0\|_2$;
\item $\omega_i =\|B_i\|_2$ for $i=0,1$, $|\lambda_0| > 1$ and $\|B_0\|_2 \leq \|B_1\|_2$;
\item $\omega_i =\|B_i\|_2$ for $i=0,1$ and  $\|B_1\|_2$ and $\|B_0\|_2$ are similar.
\end{enumerate}
\end{proposition}

\begin{proof}
First assume that $\omega_1=\omega_0= \max\{\|B_0\|_2, \|B_1\|_2\}$. From (\ref{inequality}), we get
$$\kappa_{\theta}((\lambda_0,1),L) \geq 1,$$
which implies the result by Theorem {\ref{ThKrtheta}}.
Assume now that $\omega_i =\|B_i\|_2$ for $i=0,1$. {If} $|\lambda_0| <1$ and $\|B_1\|_2 \leq \|B_0\|_2$, then
$$(1-|\lambda_0|) \|B_1\|_2 \leq (1-|\lambda_0|) \|B_0\|_2,$$
or equivalently,
$$\frac{\|B_0\|_2 + |\lambda_0| \|B_1\|_2}{\|B_1\|_2+|\lambda_0|\|B_0\|_2} \geq 1$$
and the result follows from (\ref{inequality}), which implies $\kappa_{\theta}((\lambda_0,1),L)\geq 1$, and Theorem {\ref{ThKrtheta}}.

The proof of the result assuming 3. follows from applying 2. to $rev L$ {and $\frac{1}{|\lambda_0|}$ and taking into account Lemma \ref{kappa-a-r}}.
Finally, assume that $\|B_1\|_2 \approx \|B_0\|_2$. Then, the result follows from 2. and 3.
\end{proof}
\begin{remark}We admit a certain degree of ambiguity in the meaning of ``$\|B_1\|_2$ and $\|B_0\|_2$ are similar'' in the fourth set of assumptions in Proposition \ref{comput}. It is possible to make a quantitative assumption of the type $\frac{1}{C}\leq\frac{\|B_1\|_2}{\|B_0\|_2}\leq C$ for some constant $C>1$, which would more laboriously lead to more complicated lower bounds for $\kappa_a(\lambda_0,L)$ and $\kappa_r(\lambda_0,L)$ which decrease as $C$ increases. We have preferred the simpler but somewhat ambiguous statement in Proposition \ref{comput}.
\end{remark}

Taking into account Proposition \ref{comput}, the only case left to study {of the potential non-computability of eigenvalues with small and large absolute values} is when $|\lambda_0| \ll 1$  and $\|B_0\|_2 \ll \|B_1\|_2.$ The case in which $|\lambda_0| \gg 1$ and $\|B_0\|_2 \gg \|B_1\|_2$ follows {by} using the reversal of $L(\lambda)$ and the fact that
$$\kappa_r(\lambda_0, L) = \kappa_r\left (\frac{1}{\lambda_0}, rev L \right).$$
Notice that when $\|B_0\|_2 \ll \|B_1\|_2$, in practice, $L(\lambda)$  is essentially  $\lambda B_1$, and since $L(\lambda)$ is regular, $B_1$ {is generically} nonsingular, which implies that, {when $\frac{\|B_0\|_2}{\|B_1\|_2}$ approaches 0}, all the eigenvalues of $L(\lambda)$ are very close to 0, {independently of the specific matrices $B_0$ and $B_1$. Thus, in} this trivial case, we should expect that the eigenvalues are all well-conditioned. Next we show an example that illustrates this observation.

\begin{example}
	\label{examp}
Let $\epsilon <1$. Consider the regular pencil
$$L(\lambda)= \lambda B_1 + B_0:=  \lambda\left[\begin{array}{cc}
1/\epsilon & 1/\epsilon\\
1/\epsilon & 1
\end{array}\right]+
\left[\begin{array}{cc}
\epsilon & 1\\
1 & 1
\end{array}\right].
$$

We note that the {standard condition numbers for inversion, i.e., $\|B_i\|_2\|B_i^{-1}\|_2$,} of the matrix coefficients $B_0$ and $B_1$ are {bounded} by
$$\frac{1}{1-\epsilon}\leq cond(B_0)=cond(B_1)\leq\frac{4}{1-\epsilon}$$
and
$$ \ \frac{2}{\sqrt{2}}\leq\|B_0\|_2\leq2,\quad \frac{2}{\sqrt{2}\epsilon}\leq\|B_1\|_2\leq\frac{2}{\epsilon}.$$
Thus, as $\epsilon$ approaches 0, $cond(B_0)= cond(B_1) \approx 1$ and $\|B_0\|_2 \ll \|B_1\|_2$.
Moreover, the eigenvalues of {$L(\lambda)$}, given by
$$\lambda_0=\frac{\epsilon-\epsilon^3+\sqrt{\epsilon^6-6\epsilon^4+8\epsilon^3-3\epsilon^2}}{2\epsilon-2}, \quad \lambda_1=\frac{\epsilon-\epsilon^3-\sqrt{\epsilon^6-6\epsilon^4+8\epsilon^3-3\epsilon^2}}{2\epsilon-2}$$
are both of order $\epsilon$ as $\epsilon \to 0$. Additionally, {the right (resp. left) eigenvectors of $L(\lambda)$ associated with $\lambda_0$ and $\lambda_1$ are} denoted, respectively, by $x_{\lambda_0}$ and $x_{\lambda_1}$ (resp. $y_{\lambda_0}$ and $y_{\lambda_1}$)  are given by
$$x_{\lambda_0}=\overline{y_{\lambda_0}}=\left(1, \frac{-\lambda_0-\epsilon^2}{\lambda_0+\epsilon}\right)^T, \ x_{\lambda_1}=\overline{y_{\lambda_1}}=\left(1, \frac{-\lambda_1-\epsilon^2}{\lambda_1+\epsilon}\right)^T.$$
{Observe} that, as $\epsilon \to 0$, we have
$\|x_{\lambda_0}\|_2=\|y_{\lambda_0}\|_2\approx \|y_{\lambda_1}\|_2 = \|x_{\lambda_1}\|_2 \approx \sqrt{2} \, .
$
{Moreover, we} have
\begin{align*}\kappa_{\theta}((\lambda_0,1),L)%&=\frac{(|\alpha_0|\|B_1\|_2+|\beta_0|\|B_0\|_2)\|y_{\lambda_0}\|_2\|x_{\lambda_0}\|_2}{|y_{\lambda_0}^*(\alpha_0 B_1+\beta_0 B_0)x_{\lambda_0}|}\\
&=\frac{(|\lambda_0|\|B_1\|_2+\|B_0\|_2)\|y_{\lambda_0}\|_2\|x_{\lambda_0}\|_2}{|y_{\lambda_0}^*( B_1-\overline{\lambda_0} B_0)x_{\lambda_0}|}.
%& = \frac{(|\lambda_0|\|B_1\|_2+\|B_0\|_2)\|y_{\lambda_0}\|_2\|x_{\lambda_0}\|}{|\lambda_0| |y_{\lambda_0}^*B_1x_{\lambda_0}|}= O(1), \quad \textrm{when $ \epsilon \to 0$}.
\end{align*}
Some computations show that
\begin{align*}
y_{\lambda_0}^*( B_1-\overline{\lambda_0} B_0&)x_{\lambda_0}=\frac{1}{\epsilon} - \overline{\lambda_0} \epsilon+ \left( \frac{1}{\epsilon}+ \overline{\lambda_0} \epsilon\right)\left ( \frac{-\lambda_0-\epsilon^2}{\lambda_0+\epsilon} \right) \\+
&\left[ \frac{1}{\epsilon} - \overline{\lambda_0} + \frac{(-\lambda_0 - \epsilon^2)(1+\overline{\lambda_0})}{\lambda_0+\epsilon}\right] \left(\frac{-\epsilon - \epsilon^2}{2\epsilon}\right).
\end{align*}
Combining all this information, we deduce that $\kappa_{\theta}((\lambda_0,1), L) =O( \epsilon)$ as $\epsilon \to 0$, noticing that the numerator tends to {a constant} while the denominator is $O(1/\epsilon)$.
Then, from Theorem \ref{ThKrtheta} and from $\frac{1+|\lambda_0|^2}{|\lambda_0|}=O(\frac{1}{\epsilon})$, we have that $\kappa_r(\lambda_0, L) = O(1)$. It is clear that a similar conclusion can be obtained for {$\kappa_r(\lambda_1, L)$}.

 In the {Figure} \ref{fig}, we plot the values of $|\lambda_0|$, {$\kappa_{\theta}((\lambda_0,1),L)$} and $\kappa_r(\lambda_0,L)$ against $\epsilon$. We can observe that, as $\epsilon\rightarrow 0$, both $|\lambda_0|$ and {$\kappa_{\theta}((\lambda_0,1),L)$} tend to 0 at the same rate {but that $\kappa_r(\lambda_0,L)$ remains essentially constant and close to 1}. A similar graph can be obtained for $\lambda_1$.

\begin{figure}
	\centering
	\includegraphics[width=9.3cm, height=7cm]{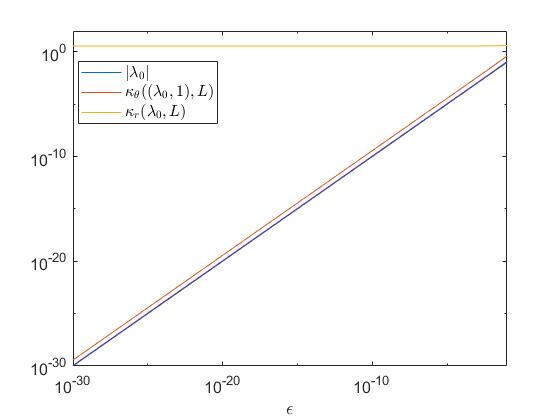}
	\caption{Plot of $|\lambda_0|$, $\kappa_{\theta}((\lambda_0,1),L)$ and $\kappa_r(\lambda_0,L)$ against $\epsilon$ in Example \ref{examp}.}
	\label{fig}
\end{figure}

\end{example}

In summary, we have proven that, as $\epsilon \to 0$, the pencil $L(\lambda)=\lambda B_1+B_0$ {in Example \ref{examp}} has well-conditioned matrix coefficients with very different norms, it has two eigenvalues close to 0 and they are both computable, since {their} {non-homogeneous relative} condition numbers {are} approximately 1. {It} is easy to check that {$rev L$} is an example of a pencil that has  computable very large eigenvalues. {As commented before, these pencils illustrate the only situation in which non-homogeneous eigenvalues with very small or large modulus can be accurately computed, which corresponds to pencils with highly unbalanced coefficients.}

%%%%%%%%%%%%%%%%%%%%%%%%%%

\section{Conclusions}
\label{secfinal}
We have gathered together the definitions of (non-homogeneous and homogeneous) eigenvalue condition numbers of matrix polynomials that were scattered in the literature. We have also derived for the first time an exact formula to compute one of these condition numbers (the homogeneous condition number that is based on the chordal distance, also called Stewart-Sun condition number). On the one hand, we have determined that the two homogeneous condition numbers studied in this paper {differ at most by a factor $\sqrt{k+1}$, where $k$ is the grade of the polynomial, and so are essentially equal in practice. Since the definition of the homogeneous condition number based on the chordal distance is considerably simpler, we believe that its use should be preferred among the homogeneous condition numbers}. On the other hand, we have proved exact relationships between each of the non-homogeneous condition numbers and the homogeneous condition number based on the chordal distance. This result would allow us to extend results that have been proved for the non-homogeneous condition numbers to the homogeneous condition numbers (and vice versa). Besides, we have provided geometric interpretations of the factor that appears in these exact relationships, which explain transparently when and why the non-homogeneous condition numbers are much larger than the homogeneous ones. Finally, we have used these relationships to analyze for which cases the large and small non-homogeneous eigenvalues of {a} matrix polynomial are computable {with some accuracy and we have seen that this is only possible in some very particular situations}.
%%%%%%%%%%%%%%%%%%%%%%%%%

\end{document}